\documentclass[reqno, 11pt, a4paper]{amsart}

\usepackage{fullpage}
\selectfont

\usepackage[utf8]{inputenc}
\usepackage[OT2,T1]{fontenc}
\usepackage[english]{babel}
\usepackage{amsmath}
\usepackage{amsfonts}
\usepackage{amssymb}
\usepackage{amsthm}
\usepackage{enumerate}
\usepackage{mathtools}
\usepackage[backref=page,pagebackref=true,hyperindex=true,bookmarks=true]{hyperref}
\usepackage{float}
\usepackage[ruled,vlined,boxed
]{algorithm2e}

\usepackage[dvipsnames]{xcolor}
\usepackage{graphicx}
\usepackage{tikz}
\usepackage{relsize}
\usepackage{stmaryrd}
\usepackage{bigints}
\usepackage[toc,page]{appendix}
\usepackage{graphicx}

\hypersetup{
  pdfauthor   = {E. Eid},
  pdftitle    = {Fast computation of hyperelliptic curve isogenies},
  pdfsubject  = {},
  pdfkeywords = {},
  colorlinks=true,
  breaklinks=true, urlcolor=blue, linkcolor=blue, citecolor=blue,
  bookmarksopen=true
}

\usetikzlibrary{cd}

\newcommand{\kt}{k \llbracket t \rrbracket}
\newcommand{\kbt}{\bar{k} \llbracket t \rrbracket}

\newcommand{\OKt}{\mathcal{O}_{K} \llbracket t  \rrbracket}

\newcommand{\OK}{\mathcal{O}_{K}}
\newcommand{\OL}{\mathcal{O}_{L}}

\newcommand{\vertIII}[1]{{\left\vert\kern-0.25ex\left\vert\kern-0.25ex\left\vert #1
    \right\vert\kern-0.25ex\right\vert\kern-0.25ex\right\vert}}
\newcommand{\softO}{\tilde O}

\newcommand{\cM}{\text{\rm M}}

\newtheorem{defi}{Definition}
\newtheorem{thm}[defi]{Theorem}
\newtheorem*{thm*}{Theorem}

\newtheorem{prop}[defi]{Proposition}

\theoremstyle{remark}

\newcommand{\newabstract}[1]{%
  \par\bigskip
  \csname captions#1\endcsname
  \item[\hskip\labelsep\scshape Disclaimer.]
}

\title{Computing isogenies between Jacobians of hyperelliptic curves of arbitrary genus via differential equations }

\author[Eid]{Elie Eid}
\address{%
  Elie Eid, %
  Univ. Rennes, %
  CNRS, IRMAR - UMR 6625, F-35000
  Rennes, %
  France. %
}
\email{elie.eid@univ-rennes1.fr}

\begin{document}

\maketitle

\begin{abstract}
Let $p$ be an odd prime number and $\ell$ be an integer coprime to $p$. We survey an algorithm for computing explicit rational representations of $(\ell , \ldots , \ell)$-isogenies between Jacobians of
  hyperelliptic curves of arbitrary genus over an extension $K$ of the field of
  $p$-adic numbers $\mathbb{Q}_p$. The algorithm has a quasi-linear complexity in $\ell$ as well as in the genus of the curves.

\end{abstract}

\section{Introduction}
\label{sec:introduction}
Over the last few years there has been a growing interest in computational aspects of abelian varieties, especially Jacobians of algebraic curves. When such a variety is given, a first task is to compute the number of points on it in some finite field~\cite{Schoof95,ballentine17}. A way to achieve this efficiently is to work with isogenies. In addition to point counting, the computation of isogenies has many applications in number theory and cryptography~\cite{CEL12,CL13,EVYan19,costello20}.  \\
In order to have optimal algorithms for computing isogenies, in particular those which are defined over finite fields, several approaches have been suggested. One of them consists in reducing the problem to the computation of a solution of a nonlinear differential equation~\cite{Elkies97,couezo15}, possibly after having lifted the problem to the $p$-adics~\cite{lesi08,lava16,careidler20,eid20}.
In this work, we focus on $p$-adic algorithms that compute the explicit form of a rational representation of an isogeny between Jacobians of hyperelliptic curves for fields of odd characteristic. 

\medskip

Let $k$ be a field of characteristic different from $2$ and $\ell >1$ and $g >1$ two integers. Let $C$ (resp. $C_1$) be a genus $g$ hyperelliptic curve over $k$ and let $J$ (resp. $J_1$) be its Jacobian. We assume that there exists a separable $(\ell , \ldots , \ell)$-isogeny $I: J_1 \rightarrow J$ defined over $k$ and we are interested in computing one of its rational representations. Let us recall briefly the definition of a rational representation and how we compute it (see~\cite{eid20} for more details). Let $P$ be a Weierstrass point on $C_1$ and $j_P : C_1 \rightarrow J_1$ the Jacobi map with origin $P$. The morphism $I \circ j_p$ induces a morphism $I_P : C_1 \rightarrow C^{(g)}$, where $C^{(g)}$ is the $g$-th symmetric power of $C$. When a coordinate system for $C_1$ and $C$ is fixed, the morphism $I_P$ is given by its Mumford representation, which consists of a pair of polynomials $(U(z), V(z))$ with the following property: if $I_P(Q)=\{ R_1 , \ldots, R_g\}$ (for some $Q \in C_1$), then $U(x({R_i})) = 0$ and $V(x({R_i}))=y({R_i})$, for all $i=1,\ldots, g$. 
Here $x(R_i)$ and $y(R_i)$ denote the coordinates of the point $R_i$.
The $2g$ coefficients of the two polynomials $U$ and $V$ can be represented as rational fractions over $k$ in one variable and they form what we call a rational representation of $I$. \\ We assume that $C$ (resp. $C_1$) is given by the affine model $C: y^2 = f(x)$ (resp. $C_1: v^2 = f_1(u)$). Let $Q$ be a non-Weierstrass point on $C_1$ such that $I_P (Q)=\{(x_1^{(0)}, y_1^{(0)}) , \ldots , (x_g^{(0)}, y_g^{(0)}) \}$ contains $g$ distinct points and does not contain neither a point at infinity nor a Weierstrass point. Let $t$ be a formal parameter of $C_1$ at $Q$ and let $\{(x_1(t), y_1(t)) , \ldots , (x_g(t), y_g(t)) \}$ be the image of $Q(t)$ by $I_P$. The action of $I_P$ on the spaces of holomorphic differentials of $C_1$ and $C^{(g)}$ gives the following differential system whose unknown is $X(t) = (x_1(t) , \ldots , x_g(t)) \in \kbt $. 
\begin{equation}
\label{eq:mainsystem}
\left\{ 
\begin{array}{l}
H(X(t)) \cdot  X'(t) = G(t)\\
y_i(t)^2 = f(x_i(t)), \, i= 1, \ldots , g \\
X(0) = (x_1^{(0)} , \cdots , x_g^{(0)}) \\
Y(0) = (y_1^{(0)} , \cdots , y_g^{(0)})
\end{array} 
\right. 
\end{equation}
where $G(t) = (G_1(t) , \ldots , G_g(t)) \in \kt ^g $ and $H(X(t))$ is the matrix defined by 

\begin{equation}
\label{eq:MatrixH}
H(x_1(t),\ldots, x_g(t)) = 
\begin{pmatrix}
x_1(t)/y_1(t) & x_2(t)/y_2(t)& \cdots & x_g(t)/y_g(t) \\

x_1(t)^2 /y_1(t) & x_2(t)^2 /y_2(t) & &  x_g(t)^2 /y_g(t) \\

\vdots & &  & \vdots \\

x_1^{g-1}(t)/y_1(t) & x_2(t)^{g-1}/y_1(t) & \cdots & x_g(t)^{g-1}/y_g(t)
\end{pmatrix}
\end{equation}
$ $ \\
Since the coefficients of $U(z)$ are rational fractions of degree at most $O(g \ell)$~\cite[Proposition~9]{eid20}, solving Equation~\eqref{eq:mainsystem} modulo $t^{O(g \ell)}$ allows to reconstruct all the components of the rational representation (note that the polynomial $V(z)$ can be recovered using the polynomial $U(z)$ and the equation of $C$). \\
Let $p$ be an odd prime number. We assume that $k$ is a finite field of characteristic $p$. Let $K$ be an unramified extension of $\mathbb{Q}_p$ such that the residue field of $K$ is $k$. 
In~\cite{eid20}, we have designed an algorithm that computes, after lifting 
Equation~\eqref{eq:mainsystem} over $K$, an approximation of its solution. 
This algorithm is based on the following Newton iteration:
\begin{equation}
\label{eq:Newtoniteration}
X_{2m}(t) = X_m(t) + H(X_m(t))^{-1} \int (G - H(X_m(t))\cdot X'_m(t)) \, dt
\end{equation}
which gives more and more accurate (for the $t$-adic distance) solutions 
of Equation~\eqref{eq:mainsystem}.
The complexity of this algorithm is quasi-linear with respect to $\ell$
but, unfortunately, it is at least quadratic in $g$ (even if we note that the matrix $H(x_1(t), \ldots , x_n(t))$ is a 
structured matrix). The main reason for this lack of efficiency is due to 
the fact that the components of the solution $X(t)$ of 
Equation~\eqref{eq:mainsystem} are power series 
over an unramified extension of degree $g$ of $K$. However, the 
rational fractions of the rational representation are defined over the 
ring of integers $\OK$ of $K$. This is where we loose an extra factor~$g$.

In this article, we revisit the algorithm of~\cite{eid20} and manage to 
lower its complexity in $g$ and make it quasi-linear as well. For this, 
we work directly on the first Mumford coordinate $U(z) = \prod \limits 
_{i=1}^g (z - x_i(t))$ which has the decisive advantage to be defined 
over the base field: 
we rewrite the Newton scheme~\eqref{eq:Newtoniteration} accordingly 
and design fast algorithms for iterating it in quasi-linear time.
Our main theorem is the following

\begin{thm}
\label{thm:main}
Let $K$ be an unramified extension of $\mathbb{Q}_p$ and $k$ its residue field. There exists an algorithm that takes  as input: 
\begin{itemize}
\item three positive integers $g$,$n$ and $N$,
\item A polynomial $f \in \OK[z]$ of degree $2g+1$, 
\item a vector $X_0 = (x_1^{(0)}, \ldots , x_g^{(0)})$ represented by the polynomial $U_0(z) = \prod \limits _{i=1}^g (z - x_i^{(0)}) \in \OK [z]$ such that, over $k$, $U_0(z)$ is separable,
\item a vector  $Y_0 = (y_1^{(0)}, \ldots , y_g^{(0)})$ represented by the interpolating polynomial $V_0(z) \in \OK [z]$ of the data $\{ (x_1^{(0)}, y_1^{(0)})  , \ldots , (x_g^{(0)}, y_g^{(0)})\}$,  
\item a vector $G(t) \in \OKt ^g$, 
\end{itemize}
and, assuming that the solution of Equation~\eqref{eq:mainsystem} has coefficients in $\OL$ with $L$ an unramified extension of $K$, outputs a polynomial $U(t,z) = \prod \limits _{i=1}^g (z - x_i(t)) \in \OKt [z]$ such that $X(t)= (x_1(t) , \ldots , x_g(t))$ is an approximation of this solution modulo $(p^N , t^{n+1})$ for a cost $\softO (ng)$ operations\footnote{The notation $\softO(-)$ means that we are hiding logarithmic factors.} in $\OK$ at precision $ O(p^M)$ with $M= {N+\lfloor \log _p (n)\rfloor}$.
\end{thm}

Important examples of isogenies are, of course, the 
multiplication-by-$\ell$ maps. Classical algorithms for computing them 
are usually based on Cantor algorithm for adding points on Jacobians 
(see for example \cite{Cantor,abela18}).
Although, they exhibit acceptable running time in practice, their 
theoretical complexity has not been well studied yet and experiments 
show that they become much slower when the genus gets higher. 
Actually, in many cases, we have observed that the algorithms 
of~\cite{eid20} perform better in practice even if their theoretical
complexity in $g$ is not optimal.
Consequently, even though the algorithms designed in the present paper 
use Kedlaya-Umans algorithm~\cite{KU11} as a subroutine and then could be difficult 
to implement in an optimized way, they appear as attractive alternatives 
for the computation of $\ell$-division polynomials on Jacobians of 
hyperelliptic curves.

\section{The main result}
In this section, we sketch the proof of the main theorem by showing that the Newton iteration given in Equation~\eqref{eq:Newtoniteration} can be executed with quasi-linear time complexity to give the desired polynomial in Theorem~\ref{thm:main}. The precision analysis has been already studied in~\cite{eid20}.  \\
Throughout this section, the letter $p$ refers to a fixed odd prime number and the letter $K$ refers to a fixed unramified extension of $\mathbb{Q}_p$ of degree $d$ and $k$ its residue field. Let $\OK$ be the ring of integers of $K$. \\
We use \textit{the fixed point arithmetic model} at precision $O(p^M)$ to do computations in $\OK$ by representing an element in $\OK$ by an expression of the form $x + O(p^M)$ with $x \in \OK/p^M \OK$. For instance, if $d=1$, the quotient $\OK/p^M \OK$ is just $\mathbb{Z}/p^M \mathbb{Z}$. Additions, multiplications and divisions in this model all reduce to the similar operations in the exact quotient ring $\OK/p^M \OK$. \\
Let $\cM(m)$ be the number of arithmetical operations required to compute the product of two polynomials of degree $m$ in an exact ring. Standard algorithms allow us to take $\cM(m) \in \softO (m)$.\\
Let $g>1$ be an integer and let $G(t) \in \OKt$. Let also $f$ be a polynomial of degree $2g+1$ and let $U_0(z)\in \OK [z]$ be a polynomial of degree $g$ which separable over $k$. For the sake of simplicity, we assume that $U_0$ is irreducible, therefore its splitting field $L$ is an unramified extension of degree $g$ of $K$. Let $x_1^{(0)}, \ldots , x_g^{(0)}$ be the roots of $U_0(z)$ in $L$ and $X_0 = (x_1^{(0)}, \ldots , x_g^{(0)})$. For $i = 1, \ldots, g$, we assume that $f(x_i^{(0)})$ has a square root $y_i^{(0)}$ in $\OL$. Take $Y_0 = (y_1^{(0)}, \ldots , y_g^{(0)})$ and let $V_0(z)\in \OK [z]$ be the interpolating polynomial of the data $\{ (x_1^{(0)}, y_1^{(0)})  , \ldots , (x_g^{(0)}, y_g^{(0)})\}$. We assume that the unique solution $X(t)=(x_1(t) , \ldots , x_n(t))$ of Equation~\eqref{eq:mainsystem} has coefficients in $\OL$ when $X_0$ and $Y_0$ are the initial conditions. \\
Let $m \in \mathbb{N}$ and $n=2m$. Let $X_m(t)=(x_1^{(m)}(t), \ldots , x_g^{(m)}(t))$ be an approximation of $X(t)$ modulo $t^m$ represented by the minimal polynomial of $x_1^{(m)}$, $U_m(t,z)= \prod \limits (z- x_i^{(m)}(t))$. We show in the next proposition that we can compute efficiently an approximation $X_n(t)=(x_1^{(n)}(t), \ldots , x_g^{(n)}(t))$ of $X(t)$ modulo $t^n$ represented by the minimal polynomial $U_n(t,z)$ of $x_1^{(n)}(t)$ using Equation~\eqref{eq:Newtoniteration}.

\begin{prop}
\label{prop:complexity}
Using the same notations as above, there exists an algorithm that computes $U_n(t,z)$ from $U_m(t,z)$ with time complexity $\softO(mg)$.
\end{prop}
\begin{proof}[Sketch of the proof]
The algorithm performs the following steps.
\begin{enumerate}
\item Compute the degree $g-1$ polynomial $W_m(t,z) = \sum \limits _{i=0}^{g-1} w_i^{(m)}(t) \,z^i$ such that 
$$W_m(t,z) \equiv 1/f^2(z) \mod (t^m, U_m(t,z))$$
and $W_m(0,z)= 1/V_0(z) \mod U_0(z)$. Observe that it is the
interpolating polynomial of the points:
$$\{ (x_1^{(m)}, 1/y_1^{(m)}) , \cdots , (x_g^{(m)},1/y_g^{(m)})\}.$$
Deduce $V_m(z) = f(z) W_m(z) \mod (t^m , U_m(z))$.\\
\item  Compute the Newton sums $s_i^{(m)}(t) = \sum \limits _{j=1} ^g ({x_j^{(m)}}(t))^i \mod t^m$ for $i=1, \ldots , 2g-1$ and deduce $r_i ^{(m)}(t) = \sum \limits _{j=1} ^g ({x_j^{(m)}(t)})^{i-1}({x_j^{(m)}(t)})' \mod t^m$. \\
\item  Compute the two products $H(X_m(t)) X'_m(t)$ and $H(X_m(t)) X_m(t)$ as follows:
$$
H(X_m(t)) X'_m(t) = 
\begin{pmatrix}
r_1^{(m)} & r_2^{(m)} & \cdots & r_g^{(m)}\\
r_2^{(m)} & r_3^{(m)} &  & r_{g+1}^{(m)} \\
\vdots \\
r_{g}^{(m)} & r_{g+1}^{(m)} & \cdots & r_{2g-1}^{(m)}
\end{pmatrix}
\begin{pmatrix}
w_0^{(m)} \\
w_1^{(m)} \\
\vdots \\
w_{g-1}^{(m)}
\end{pmatrix} \mod t^m 
$$
and 
$$
H(X_m(t)) X_m(t) = 
\begin{pmatrix}
s_1^{(m)} & s_2^{(m)} & \cdots & s_g^{(m)}\\
s_2^{(m)} & s_3^{(m)} &  & s_{g+1}^{(m)} \\
\vdots \\
s_{g}^{(m)} & s_{g+1}^{(m)} & \cdots & s_{2g-1}^{(m)}
\end{pmatrix}
\begin{pmatrix}
w_0^{(m)} \\
w_1^{(m)} \\
\vdots \\
w_{g-1}^{(m)}
\end{pmatrix} \mod t^m
$$ $ $ \\
\item  Compute $(F_1^{(m)}, \cdots , F_g^{(m)}) = H(X_m(t)) X_m(t) - \int (G(t) - H(X_m(t))X'_m(t) ) \, dt$. \\
\item  Let $D_m(t,z) =F_1^{(m)} z^g + F_2^{(m)} z^{g-1} + \ldots + F_{g-1}^{(m)}z^2 + F_g^{(m)}z$. Compute $U_m(t,z)D_m(t,z)= q_{2g}^{(m)}z^{2g} + q_{2g-1}^{(m)}z^{2g-1} + \ldots  + q_0^{(m)} \mod t^m$ and read off the polynomial $Q_m(t,z) =q_{2g}^{(m)}z^{g-1} + q_{2g-1}^{(m)}z^{g-2} + \ldots  + q_{g+1}^{(m)} $.\\
\item Compute $T_m(t,z) = \dfrac{Q_m(t,z) V_m(t,z)}{U'_m(t,z)} \mod (t^m, U_m(t,z))$. \\
\item  Compute $U_n(t,z)$ such that $U_n(t,T_m(t,z)) \equiv 0 \mod (t^m,U_m(t,z))$. \\
$ $ 
\end{enumerate}
We now discuss briefly the complexity analysis.
The polynomial $W_m$ in step 1 can be efficiently computed by the classical Newton scheme for extracting square roots.
Since, the coefficients of $W_m$ and $V_m$ are polynomials of degrees at most $m$ defined over $K$, the complexity of this step is $O(\cM(m)\cM(g))$.
The computation of the Newton sums $s_i^{(m)}$ of $U_m$ in step 2 is 
classical~\cite{NewtonSums} and can be carried out for a cost of 
$O(\cM(m)\cM(g))$ operations. 
In step 3, we are dealing with two Hankel 
matrix-vector products. This can be done in $O(\cM(m)\cM(g))$ operations 
in $K$~\cite[Section~3a]{Canny}. The polynomial $T_m$ constructed in 
step 5 and step 6 interpolates the data $\{ (x_1^{(m)} , x_1^{(n)}) , 
\ldots , (x_g^{(m)} , x_g^{(n)}) \}$ (see \cite[Section~5]{Kaltofen} for 
more details), it can be computed in $O(\cM(m)\cM(g))$ as well. Step 7 
computes $U_n$, the minimal polynomial of $x_1^{(n)}$. We make use of 
Kedlaya-Umans algorithm~\cite{KU11} to execute step 7; the resulting
bit complexity is $\softO(mg)$.
\end{proof}

\bibliographystyle{alphaabbr}
\bibliography{synthbib}

\end{document}